\documentclass[a4paper,11pt]{amsart}
\usepackage{graphicx}
\usepackage{mathptmx}
\usepackage{amsmath}
\usepackage{amssymb}
\usepackage{enumitem}
\usepackage{xcolor}
\usepackage{xparse}
\NewDocumentCommand{\eulerian}{omm}
 {%
  \genfrac<>{0pt}{}{#2}{#3}%
  \IfValueT{#1}{_{\!#1}}%
 }

 \newmuskip\pFqmuskip

\newcommand*\pFq[6][8]{%
  \begingroup 
  \pFqmuskip=#1mu\relax
  \mathchardef\normalcomma=\mathcode`,
  \mathcode`\,=\string"8000
  \begingroup\lccode`\~=`\,
  \lowercase{\endgroup\let~}\pFqcomma
  {}_{#2}F_{#3}{\left(\genfrac..{0pt}{}{#4}{#5}\bigg|#6\right)}%
  \endgroup
}
\newcommand{\pFqcomma}{{\normalcomma}\mskip\pFqmuskip}

\newtheorem{theorem}{Theorem}
\newtheorem{lemma}[theorem]{Lemma}
\newtheorem{corollary}[theorem]{Corollary}
\newtheorem{proposition}[theorem]{Proposition}

\newtheorem{remark}[theorem]{Remark}

\begin{document}

\title[Some new properties on degenerate Bell polynomials]{Some new properties on degenerate Bell polynomials}

\author{Taekyun  Kim}
\address{Department of Mathematics, Kwangwoon University, Seoul 139-701, Republic of Korea}
\email{tkkim@kw.ac.kr}

\author{DAE SAN KIM}
\address{Department of Mathematics, Sogang University, Seoul 121-742, Republic of Korea}
\email{dskim@sogang.ac.kr}

\author{Hyunseok Lee}
\address{Department of Mathematics, Kwangwoon University, Seoul 139-701, Republic of Korea}
\email{luciasconstant@kw.ac.kr}

\author{SeongHo Park}
\address{Department of Mathematics, Kwangwoon University, Seoul 139-701, Republic of Korea}
\email{abcd2938471@kw.ac.kr}

\subjclass[2010]{11B73; 11B83}
\keywords{degenerate Bell polynomials; degenerate exponential functions; degenerate Stirling numbers of the first kind; degenerate Stirling numbers of the second kind}

\maketitle

\begin{abstract}
The aim of this paper is to study the degenerate Bell numbers and polynomials which are degenerate versions of the Bell numbers and polynomials. We derive some new identities and properties of those numbers and polynomials that are associated with the degenerate Stirling numbers of both kinds.
\end{abstract}

\section{Introduction}
The Bell number $B_n$ counts  the number of partitions of a set with $n$ elements into disjoint nonempty subsets. The Bell polynomials, also called Touchard or exponential polynomials, are natural extensions of Bell numbers. As a degenerate version of these Bell polynomials and numbers, the degenerate Bell polynomials and numbers (see \eqref{12}) are introduced and studied under the different names of the partially degenerate Bell polynomials and numbers in [12].\par
In [5], Carlitz initiated the exploration of degenerate Bernoulli and Euler polynomials, which are degenerate versions of the ordinary Bernoulli and Euler polynomials. Along the same line as Carlitz's pioneering work, intensive studies have been done for degenerate versions of quite a few special polynomials and numbers (see [1, 5, 7--13] and the references therein). It is worthwhile to mention that these studies of degenerate versions have been done not only for some special numbers and polynomials but also for transcendental functions like gamma functions (see [10]). The studies have been carried out by various means like combinatorial methods, generating functions, differential equations, umbral calculus techniques, $p$-adic analysis and probability theory. \par
The aim of this paper is to further investigate the degenerate Bell polynomials and numbers by means of generating functions. In more detail, we derive several properties and identities of those numbers and polynomials which include recurrence relations for degenerate Bell polynomials (see Theorems 2, 4, 5, 12), expressions for them that can be derived from repeated applications of certain operators to the exponential functions (see Theorem 3, Proposition 10), the derivatives of them (Corollary 7), the antiderivatives of them (see Theorem 9), and some identities involving them (see Theorems 8. 13). For the rest of this section, we recall some necessary facts that are needed throughout this paper.

 \vspace{0.1in}

For any $\lambda\in\mathbb{R}$, the degenerate exponential functions are defined by 
\begin{equation}
e_{\lambda}^{x}(t)=\sum_{k=0}^{\infty}\frac{(x)_{k,\lambda}}{k!}t^{k},\quad (\mathrm{see}\ [11]),\label{1}
\end{equation}
where
\begin{equation}
(x)_{0,\lambda}=1, (x)_{n,\lambda}=x(x-\lambda)\cdots(x-(n-1)\lambda),\ (n\ge 1).\label{2}
\end{equation}
When $x=1$, we see use the notation $e_{\lambda}(t)=e^{1}_{\lambda}(t)$.  \par 
In [5], Carlitz introduced the degenerate Bernoulli numbers given by 
\begin{equation}
\frac{t}{e_{\lambda}(t)-1}=\sum_{n=0}^{\infty}\beta_{n,\lambda}\frac{t^{n}}{n!}. 	\label{3}
\end{equation}
Note that $\lim_{\lambda\rightarrow 0}\beta_{n,\lambda}=B_{n}$, where $B_{n}$ are the ordinary Bernoulli numbers given by 
\begin{equation}
\frac{t}{e^{t}-1}=\sum_{n=0}^{\infty}B_{n}\frac{t^{n}}{n!},\quad(\mathrm{see}\ [1-14]).\label{4}
\end{equation} \par
It is well known that the Stirling numbers of the first kind are defined by 
\begin{equation}
(x)_{n}=\sum_{k=0}^{n}S_{1}(n,k)x^{k}, \quad(\mathrm{see}\ [14]),\label{5}
\end{equation}
where $(x)_{0}=1, (x)_{n}=x(x-1)\cdots(x-n+1),\ (n\ge 1)$. \\
As the inversion formula of \eqref{5}, the Stirling numbers of the second kind are given by 
\begin{equation}
x^{n}=\sum_{k=0}^{n}S_{2}(n,k)(x)_{k},\quad (\mathrm{see}\ [14]).\label{6}
\end{equation} \par
Recently, the degenerate Stirling numbers of the first kind are defined by 
\begin{equation}
(x)_{n}=\sum_{k=0}^{n}S_{1,\lambda}(n,k)(x)_{k,\lambda} ,\quad(\mathrm{see}\ [7]),\label{7}
\end{equation}
and the degenerate Stirling numbers of the second kind are given by 
\begin{equation}
(x)_{n,\lambda}=\sum_{k=0}^{n}S_{2,\lambda}(n,k)(x)_{k}, \quad(\mathrm{see}\ [7,9]).\label{8}
\end{equation}
We also recall the degenerate absolute Stirling numbers of the first kind that are defined by 
\begin{equation}
\langle x\rangle_{n}=\sum_{k=0}^{n}{n \brack k}_{\lambda}\langle x\rangle_{k,\lambda},\quad(\mathrm{see}\ [13]),\label{9}	
\end{equation}
where 
\begin{align*}
&\langle x\rangle_{0}=1,\quad \langle x\rangle_{n}=x(x+1)\cdots(x+n-1),\ (n\ge 1),  \\
&\langle x\rangle_{0,\lambda}=1,\quad \langle x\rangle_{n,\lambda} =x(x+\lambda)(x+2\lambda)\cdots\big(x+(n-1)\lambda
\big),\ (n\ge 1). 
\end{align*} \par
It is well known that the Bell polynomials are defined by  
\begin{equation}
e^{x(e^{t}-1)}=\sum_{n=0}^{\infty}\mathrm{Bel}_{n}(x)\frac{t^{n}}{n!},\quad (\mathrm{see}\ [3,4]).\label{10}
\end{equation}
When $x=1,\ \mathrm{Bel}_{n}= \mathrm{Bel}_{n}(1)$ are called the Bell numbers. \\
From \eqref{8}, we note that 
\begin{equation}
\mathrm{Bel}_{n}(x)=\sum_{k=0}^{n}S_{2}(n,k)x^{k},\quad(\mathrm{see}\ [3,4]).\label{11}
\end{equation} \par
In [12], the degenerate Bell polynomials are defined by 
\begin{equation}
e^{x(e_{\lambda}(t)-1)}=\sum_{n=0}^{\infty} \mathrm{Bel}_{n,\lambda}(x)\frac{t^{n}}{n!}. \label{12}
\end{equation}
Note that $\lim_{\lambda\rightarrow 0} \mathrm{Bel}_{n,\lambda}(x)=\mathrm{Bel}_{n}(x)$. For $x=1$, $\mathrm{Bel}_{n,\lambda}=\mathrm{Bel}_{n,\lambda}(1)$ are called the degenerate Bell numbers. \par 
The compositional inverse of $e_{\lambda}(t)$ is given by $\log_{\lambda}(t)$, namely $e_{\lambda}(\log_{\lambda}(t))=t=\log_{\lambda}(e_{\lambda}(t))$,\\
where
\begin{equation}
\log_{\lambda}(1+t)=\frac{1}{\lambda}\big((1+t)^{\lambda}-1\big)=\sum_{n=1}^{\infty}\lambda^{n-1}(1)_{n,\frac{1}{\lambda}}\frac{t^{n}}{n!},\quad(\mathrm{see}\ [7]).\label{13} 
\end{equation}
Note that $\displaystyle\lim_{\lambda\rightarrow 0}\log_{\lambda}(1+t)=\log(1+t)\displaystyle$.\par 
From \eqref{12}, we note that 
\begin{align}
\mathrm{Bel}_{n,\lambda}(x)\ =\ e^{-x}\sum_{k=0}^{\infty}\frac{(k)_{n,\lambda}}{k!}x^{k} 
=\ \sum_{k=0}^{n}S_{2,\lambda}(n,k)x^{k},\quad (\mathrm{see}\ [12]). \label{14}
\end{align}

\section{Some new properties on degenerate Bell polynomials}

Let $a$ be a nonzero constant. First, we observe that 
\begin{align}
\frac{d^{n}}{dt^{n}}e^{a(e_{\lambda}(t))}&=\frac{d^{n}}{dt^{n}}\sum_{k=0}^{\infty}\frac{a^{k}}{k!}e_{\lambda}^{k}(t)=\sum_{k=0}^{\infty}\frac{a^{k}}{k!}(k)_{n,\lambda}e_{\lambda}^{k-n\lambda}(t)\label{15} \\
&=\sum_{k=0}^{\infty}\frac{(k)_{n,\lambda}}{k!}a^{k}e_{\lambda}^{k}(t)\frac{1}{(1+\lambda t)^{n}} \nonumber \\
&=\bigg(\sum_{k=0}^{\infty}\frac{(k)_{n,\lambda}}{k!}\big(ae_{\lambda}(t)\big)^{k}e^{-ae_{\lambda}(t)}\bigg)e^{ae_{\lambda}(t)}\frac{1}{(1+\lambda t)^{n}}\nonumber \\
&= \frac{1}{(1+\lambda t)^{n}}\mathrm{Bel}_{n,\lambda}\big(ae_{\lambda}(t)\big)e^{ae_{\lambda}(t)}.\nonumber
\end{align}
Therefore, by \eqref{15}, we obtain the following lemma. 
\begin{lemma}
For $n\ge 0$, we have 
\begin{displaymath}
	\frac{d^{n}}{dt^{n}}e^{a(e_{\lambda}(t))} = \frac{1}{(1+\lambda t)^{n}}\mathrm{Bel}_{n,\lambda}\big(ae_{\lambda}(t)\big)e^{ae_{\lambda}(t)}.
\end{displaymath}
\end{lemma}
Let $x=e_{\lambda}(t)$ in \eqref{15}. Then we have 
\begin{equation}
\frac {d}{dt}=\frac{dx}{dt}\frac{d}{dx}=\frac{1}{1+\lambda t}e_{\lambda}(t)\frac{d}{dx}=x^{1-\lambda}\frac{d}{dx}. \label{16}
\end{equation}
By Lemma 1 and \eqref{16}, we get 
\begin{equation}
\bigg(x^{1-\lambda}\frac{d}{dx}\bigg)^{n}e^{ax}=x^{-n\lambda}\mathrm{Bel}_{n,\lambda}(ax)e^{ax},\quad (n\ge 0).\label{17}	
\end{equation} \par
Let 
\begin{equation}
S_{n,\lambda}=\sum_{k=0}^{\infty}\frac{(k)_{n,\lambda}}{k!},\quad n=0,1,2,\dots.\label{18}
\end{equation}
Then we note from \eqref{14} that we have
\begin{align}
e\mathrm{Bel}_{n,\lambda}=S_{n,\lambda}.\label{19}
\end{align}
The generating function of $S_{n,\lambda}$ is given by 
\begin{equation*}
e^{e_{\lambda}(t)}=\sum_{n=0}^{\infty}S_{n,\lambda}\frac{t^{n}}{n!}. 
\end{equation*}
Indeed, this can be seen from the following:
\begin{align}\label{20}	
\sum_{n=0}^{\infty}S_{n,\lambda}\frac{t^{n}}{n!}=e \sum_{n=0}^{\infty}\mathrm{Bel}_{n,\lambda}(1)\frac{t^{n}}{n!}
=ee^{e_{\lambda}(t)-1}=e^{e_{\lambda}(t)}.
\end{align} \par
Taking the derivative with respect to $t$ on both sides of \eqref{20}, we have 
\begin{align}
\sum_{n=0}^{\infty}S_{n+1,\lambda}\frac{t^{n}}{n!}&=\frac{d}{dt}e^{e_{\lambda}(t)}=e_{\lambda}^{1-\lambda}(t)e^{e_{\lambda}(t)}
=\sum_{l=0}^{\infty}(1-\lambda)_{l,\lambda}\frac{t^{l}}{l!}\sum_{m=0}^{\infty}S_{m,\lambda}\frac{t^{m}}{m!}\label{21}\\
&=\sum_{n=0}^{\infty}\bigg(\sum_{m=0}^{n}\binom{n}{m}S_{m,\lambda}(1-\lambda)_{n-m,\lambda}\bigg)\frac{t^{n}}{n!} \nonumber  \\
&=\sum_{n=0}^{\infty}\bigg(\sum_{m=0}^{n}\binom{n}{m}S_{m,\lambda}(1)_{n-m+1,\lambda}\bigg)\frac{t^{n}}{n!}. \nonumber
\end{align}
Thus, by comparing the coefficients on both sides of \eqref{21} and from \eqref{19}, we obtain the following theorem. 
\begin{theorem}
For $n\ge 0$, we have 
\begin{displaymath}
\mathrm{Bel}_{n+1,\lambda}= \sum_{m=0}^{n}\binom{n}{m}\mathrm{Bel}_{m,\lambda}(1)_{n-m+1,\lambda}.
\end{displaymath}
\end{theorem}
Assume that  the following identity holds:
\begin{align*}
\bigg(x^{1-\lambda}\frac{d}{dx}\bigg)^{n}e^{x}=\sum_{k=0}^{\infty}\frac{(k)_{n,\lambda}}{k!}x^{k-n\lambda}.
\end{align*}
Then we have
\begin{align*}
&\bigg(x^{1-\lambda}\frac{d}{dx}\bigg)^{n+1}e^{x}=x^{1-\lambda}\frac{d}{dx}\sum_{k=0}^{\infty}\frac{(k)_{n,\lambda}}{k!}x^{k-n\lambda} \\
&=x^{1-\lambda}\sum_{k=0}^{\infty}\frac{(k)_{n,\lambda}}{k!}(k-n \lambda)x^{k-n\lambda-1}
=\sum_{k=0}^{\infty}\frac{(k)_{n+1,\lambda}}{k!}x^{k-(n+1)\lambda}.
\end{align*}
This together with \eqref{17} gives the next result.
\begin{theorem}
For $n\ge 0$, we have 	
\begin{equation}
\bigg(x^{1-\lambda}\frac{d}{dx}\bigg)^{n}e^{x}=\sum_{k=0}^{\infty}\frac{(k)_{n,\lambda}}{k!}x^{k-n\lambda}=x^{-n\lambda}\mathrm{Bel}_{n,\lambda}(x)e^{x}. \label{22}
\end{equation}
\end{theorem}
From the first equality in \eqref{22} and \eqref{18}, we see that we have
\begin{align}
S_{n,\lambda}=\bigg(x^{1-\lambda}\frac{d}{dx}\bigg)^{n}e^{x}\big|_{x=1}.\label{23}
\end{align}
Clearly, $S_{0,\lambda}=S_{1,\lambda}=e$. We can check that
\begin{align}
&\big(x^{1-\lambda}\frac{d}{dx}\big)^{2}e^{x}=
(1-\lambda)x^{1-2\lambda}e^{x}+x^{2-2\lambda}e^{x}, \label{24}\\
&\bigg(x^{1-\lambda}\frac{d}{dx}\bigg)^{3}e^{x}=x^{1-3\lambda}e^{x}\Big((1-\lambda)_{2,\lambda}+(1-\lambda)x\Big)+x^{2-3\lambda}e^{x}(2-2\lambda+x).\nonumber
\end{align}
From \eqref{23} and \eqref{24}, we have $S_{2,\lambda}=(2-\lambda)e,\quad S_{3,\lambda}=(2\lambda^{2}-6\lambda+5)e$.

\vspace{0.1in}

By taking $x\frac{d}{dx}$ in the second equality of \eqref{22}, on the one hand we have 
\begin{equation}
x\frac{d}{dx}\bigg(x^{-n\lambda}\mathrm{Bel}_{n,\lambda}(x)e^{x}\bigg)=\sum_{k=0}^{\infty}\frac{(k)_{n+1,\lambda}}{k!}x^{k-n\lambda}. \label{25}	
\end{equation}
On the other hand, we also have
\begin{align}\label{26}
x\frac{d}{dx}\Big(x^{-n\lambda}\mathrm{Bel}_{n,\lambda}(x)e^{x}\Big)=xx^{-n\lambda}\Big(\mathrm{Bel}_{n,\lambda}^{\prime}(x)+\mathrm{Bel}_{n,\lambda}(x)\Big)e^{x}-n\lambda x^{-n\lambda}\mathrm{Bel}_{n,\lambda}(x)e^{x},
\end{align}
where $\mathrm{Bel}^{\prime}_{n,\lambda}(x)=\frac{d}{dx}\mathrm{Bel}_{n,\lambda}(x)$. \\ 
From \eqref{25}, \eqref{26} and Theorem 3, we note that 
\begin{align}
\sum_{k=0}^{\infty}\frac{(k)_{n+1,\lambda}}{k!}x^{k}&=\bigg(\sum_{k=0}^{\infty}\frac{(k)_{n+1,\lambda}}{k!}x^{k}e^{-x}\bigg)e^{x}=\mathrm{Bel}_{n+1,\lambda}(x)e^{x}\label{27}\\
&=x\Big(\mathrm{Bel}_{n,\lambda}^{\prime}(x)+\mathrm{Bel}_{n,\lambda}(x)\Big)e^{x}-n\lambda\mathrm{Bel}_{n,\lambda}(x)e^{x}. \nonumber
\end{align}
Therefore, by \eqref{27} and Theorem 3, we obtain the following theorem. 
\begin{theorem}
For $n\ge 0$, we have 
\begin{displaymath}
\mathrm{Bel}_{n+1,\lambda}(x)=	x\Big(\mathrm{Bel}_{n,\lambda}^{\prime}(x)+\mathrm{Bel}_{n,\lambda}(x)\Big)-n\lambda\mathrm{Bel}_{n,\lambda}(x),
\end{displaymath}
where $\mathrm{Bel}^{\prime}_{n,\lambda}(x)=\frac{d}{dx}\mathrm{Bel}_{n,\lambda}(x)$.
\end{theorem}
From \eqref{12}, we note that 
\begin{align}
\sum_{n=0}^{\infty}\frac{d}{dx}\mathrm{Bel}_{n,\lambda}(x)\frac{t^{n}}{n!}&=\frac{\partial}{\partial x}e^{x(e_{\lambda}(t)-1)}=(e_{\lambda}(t)-1)e^{x(e_{\lambda}(t)-1)}\label{28} \\
&=\bigg(\sum_{l=0}^{\infty}(1)_{l,\lambda}\frac{t^{l}}{l!}-1\bigg)\sum_{m=0}^{\infty}\mathrm{Bel}_{m,\lambda}(x)\frac{t^{m}}{m!}\nonumber \\
&=\sum_{n=0}^{\infty}\bigg(\sum_{m=0}^{n}\binom{n}{m}\mathrm{Bel}_{m,\lambda}(x)(1)_{n-m,\lambda}-\mathrm{Bel}_{n,\lambda}(x)\bigg)\frac{t^{n}}{n!}\nonumber \\
&=\sum_{n=0}^{\infty}\bigg(\sum_{m=0}^{n-1}\binom{n}{m}\mathrm{Bel}_{m,\lambda}(x)(1)_{n-m,\lambda}\bigg)\frac{t^{n}}{n!}.\nonumber
\end{align}
Thus, by comparing the coefficients on both sides of \eqref{28}, we get 
\begin{equation}
\frac{d}{dx}\mathrm{Bel}_{n,\lambda}(x)=\mathrm{Bel}^{\prime}_{n,\lambda}(x)=	\sum_{m=0}^{n-1}\binom{n}{m}\mathrm{Bel}_{m,\lambda}(x)(1)_{n-m,\lambda},\quad (n\ge 1).\label{29}
\end{equation} \par
Taking the derivative with respect to $t$ on both sides of \eqref{12}, we have 
\begin{equation}
\frac{d}{dt}e^{x(e_{\lambda}(t)-1)}=\sum_{n=0}^{\infty}\mathrm{Bel}_{n+1,\lambda}(x)\frac{t^{n}}{n!}. \label{30}
\end{equation}
On the other hand, 
\begin{align}
\frac{d}{dt}e^{x(e_{\lambda}(t)-1)}&=xe_{\lambda}^{1-\lambda}(t) e^{x(e_{\lambda}(t)-1)}=x \sum_{l=0}^{\infty}(1-\lambda)_{l,\lambda}\frac{t^{l}}{l!}\sum_{m=0}^{\infty}\mathrm{Bel}_{m,\lambda}(x)\frac{t^{m}}{m!} \label{31}\\
&=x\sum_{n=0}^{\infty}\bigg(\sum_{m=0}^{n}\binom{n}{m}\mathrm{Bel}_{n,\lambda}(x)(1-\lambda)_{n-m, \lambda}\bigg)\frac{t^{n}}{n!} \nonumber \\
&=\sum_{n=0}^{\infty}\bigg(x\sum_{m=0}^{n}\binom{n}{m}\mathrm{Bel}_{m,\lambda}(x)(1)_{n-m+1,\lambda}\bigg)\frac{t^{n}}{n!}. \nonumber
\end{align}
Therefore, by \eqref{30} and \eqref{31}, we obtain the following theorem. 
\begin{theorem}
For $n\ge 0$, we have 
\begin{displaymath}
\mathrm{Bel}_{n+1,\lambda}(x)= x\sum_{m=0}^{n}\binom{n}{m}\mathrm{Bel}_{m,\lambda}(x)(1)_{n-m+1,\lambda}.
\end{displaymath}	
\end{theorem}

\begin{remark}
Theorems 4 and 5, and \eqref{29} give us the following:
\begin{align*}
\mathrm{Bel}_{n+1,\lambda}(x)&=x\sum_{m=0}^{n}\binom{n}{m}\mathrm{Bel}_{m,\lambda}(x)(1)_{n-m,\lambda}-n\lambda\mathrm{Bel}_{n,\lambda}(x) \\
&=x\sum_{m=0}^{n}\binom{n}{m}\mathrm{Bel}_{m,\lambda}(x)(1)_{n-m,\lambda}\big(1-(n-m)\lambda \big).
\end{align*}
This implies that the following identity must hold true:
\begin{align*}
n\mathrm{Bel}_{n,\lambda}(x)= x\sum_{m=0}^{n}\binom{n}{m}(n-m)\mathrm{Bel}_{m,\lambda}(x)(1)_{n-m,\lambda},
\end{align*}
the validity of which follows from Theorem 5.
\end{remark}
From Theorem 4, we note that 
\begin{align}
x\mathrm{Bel}_{n,\lambda}^{\prime}(x)&=x\frac{d}{dx} \mathrm{Bel}_{n,\lambda}(x)= \mathrm{Bel}_{n+1,\lambda}(x)-x\mathrm{Bel}_{n,\lambda}(x)+n\lambda \mathrm{Bel}_{n,\lambda}(x)\label{32} \\
&= \mathrm{Bel}_{n+1,\lambda}(x)-(x-n\lambda) \mathrm{Bel}_{n,\lambda}(x) \nonumber \\
&=x\sum_{m=0}^{n-1}\binom{n}{m} \mathrm{Bel}_{m,\lambda}(x)(1)_{n+1-m,\lambda}+n\lambda	\mathrm{Bel}_{n,\lambda}(x).\nonumber
\end{align}
Therefore, by \eqref{32}, we obtain the following corollary. 
\begin{corollary}
For $n\ge 1$, we have 
\begin{displaymath}
x\frac{d}{dx} \mathrm{Bel}_{n,\lambda}(x)= x\sum_{m=0}^{n-1}\binom{n}{m} \mathrm{Bel}_{m,\lambda}(x)(1)_{n+1-m,\lambda}+n\lambda\mathrm{Bel}_{n,\lambda}(x).
\end{displaymath}
\end{corollary}
We observe that 
\begin{align}
&x^{1-\lambda}\frac{d}{dx}\bigg(x^{-n\lambda}\mathrm{Bel}_{n,\lambda}(x)e^{x}\bigg)=x^{1-\lambda}\frac{d}{dx}\bigg(x^{-n\lambda}\sum_{k=0}^{\infty}\frac{(k)_{n,\lambda}}{k!}x^{k}\bigg) \label{33}\\
&=\sum_{k=0}^{\infty}\frac{(k)_{n+1,\lambda}}{k!}x^{k-(n+1)\lambda}=x^{-(n+1)\lambda}\bigg(\sum_{k=0}^{\infty}\frac{(k)_{n+1,\lambda}}{k!}x^{k}e^{-x}\bigg)e^{x}\nonumber \\
&= x^{-(n+1)\lambda}\mathrm{Bel}_{n+1,\lambda}(x)e^{x},\quad (n\ge 0). \nonumber	
\end{align}
Thus, by \eqref{33}, we get 
\begin{equation}
x^{1-\lambda}\frac{d}{dx}\Big(x^{-n\lambda}\mathrm{Bel}_{n,\lambda}(x)e^{x}\bigg)=x^{-(n+1)\lambda}\mathrm{Bel}_{n+1,\lambda}(x)e^{x},\quad (n\ge 0).\label{34}
\end{equation} \par
From \eqref{12}, we have
\begin{align}
\sum_{n=0}^{\infty}\mathrm{Bel}_{n,\lambda}(x+y)\frac{t^{n}}{n!}&=	e^{(x+y)(e_{\lambda}(t)-1)}=e^{x(e_{\lambda}(t)-1)}\cdot e^{y(e_{\lambda}(t)-1)} \label{35}\\
&=\sum_{l=0}^{\infty}\mathrm{Bel}_{l,\lambda}(x)\frac{t^{l}}{l!}\sum_{m=0}^{\infty} \mathrm{Bel}_{m,\lambda}(x)\frac{t^{m}}{m!}\nonumber \\
&=\sum_{n=0}^{\infty}\bigg(\sum_{l=0}^{n}\binom{n}{l}\mathrm{Bel}_{l,\lambda}(x)\mathrm{Bel}_{n-l,\lambda}(y)\bigg)\frac{t^{n}}{n!}. \nonumber
\end{align}
Therefore, by comparing the coefficients on both sides of \eqref{35}, we obtain the following theorem. 
\begin{theorem}
For $n\ge 0$, we have 
\begin{displaymath}
\mathrm{Bel}_{n,\lambda}(x+y)= \sum_{l=0}^{n}\binom{n}{l}\mathrm{Bel}_{l,\lambda}(x)\mathrm{Bel}_{n-l,\lambda}(y).
\end{displaymath}	
\end{theorem}
From \eqref{12}, we note that 
\begin{equation}
\sum_{n=0}^{\infty}\int_{0}^{x}\mathrm{Bel}_{n,\lambda}(x)dx\frac{t^{n}}{n!}=\int_{0}^{x}e^{x(e_{\lambda}(t)-1)}dx.\label{36}	
\end{equation}
On the other hand, we also have
\begin{align}
&\int_{0}^{x}e^{x(e_{\lambda}(t)-1)}dx=\frac{1}{e_{\lambda}(t)-1}\Big[e^{x(e_{\lambda}(t)-1}\Big]_{0}^{x}\label{37} \\
&=\frac{1}{e_{\lambda}(t)-1}\Big(e^{x(e_{\lambda}(t)-1)}-1\Big)=\frac{1}{e_{\lambda}(t)-1}\sum_{k=1}^{\infty}\mathrm{Bel}_{k,\lambda}(x)\frac{t^{k}}{k!}\nonumber \\
&=\frac{t}{e_{\lambda}(t)-1}\sum_{k=0}^{\infty}\frac{\mathrm{Bel}_{k+1,\lambda}(x)}{k+1}\frac{t^{k}}{k!}=\sum_{l=0}^{\infty}\beta_{l,\lambda}\frac{t^{l}}{l!} \sum_{k=0}^{\infty}\frac{\mathrm{Bel}_{k+1,\lambda}(x)}{k+1}\frac{t^{k}}{k!}\nonumber\\
&=\sum_{n=0}^{\infty}\bigg(\sum_{k=0}^{n}\binom{n}{k}\frac{\mathrm{Bel}_{k+1,\lambda}(x)}{k+1}\beta_{n-k,\lambda}\bigg)\frac{t^{n}}{n!}\nonumber \\
&=\sum_{n=0}^{\infty}\bigg(\frac{1}{n+1}\sum_{k=0}^{n}\binom{n+1}{k+1}\mathrm{Bel}_{k+1,\lambda}(x)\beta_{n-k,\lambda}\bigg)\frac{t^{n}}{n!} \nonumber\\
&=\sum_{n=0}^{\infty}\bigg(\frac{1}{n+1}\sum_{k=1}^{n+1}\binom{n+1}{k}\mathrm{Bel}_{k,\lambda}(x)\beta_{n+1-k,\lambda}\bigg)\frac{t^{n}}{n!}. \nonumber
\end{align}
Therefore, by \eqref{36} and \eqref{37}, we obtain the following theorem.
\begin{theorem}
For $n\ge 0$, we have 
\begin{displaymath}
\int_{0}^{x}\mathrm{Bel}_{n,\lambda}(x)dx= \frac{1}{n+1}\sum_{k=1}^{n+1}\binom{n+1}{k}\beta_{n+1-k,\lambda}\mathrm{Bel}_{k,\lambda}(x),
\end{displaymath}
where $\beta_{n,\lambda}$ are the Carlitz's degenerate Bernoulli numbers given by $\frac{t}{e_{\lambda}(t)-1}=\sum_{n=0}^{\infty}\beta_{n,\lambda}\frac{t^{n}}{n!}$.
\end{theorem}
For $k\ge 0$, by \eqref{8}, we get 
\begin{align}
\sum_{n=k}^{\infty}S_{2,\lambda}(n,k)\frac{t^{n}}{n!}&=\frac{1}{k!}\Big(e_{\lambda}(t)-1\Big)^{k}=\frac{1}{k!}\sum_{j=0}^{k}(-1)^{k-j}e_{\lambda}^{j}(t)\binom{k}{j}\label{38} \\
&=\sum_{n=0}^{\infty}\bigg(\frac{1}{k!}\sum_{j=0}^{k}\binom{k}{j}(-1)^{k-j}(j)_{n,\lambda}\bigg)\frac{t^{n}}{n!}.\nonumber
\end{align}
By comparing the coefficients on both sides of \eqref{38}, we have 
\begin{equation}
\frac{1}{k!}\sum_{j=0}^{k}\binom{k}{j}(-1)^{k-j}(j)_{n,\lambda}=\left\{\begin{array}{ccc}
S_{2,\lambda}(n,k), & \textrm{if $n\ge k$,} \\
0, & \textrm{if $0 \le n \le k-1$.}
\end{array}\right.\label{39}	
\end{equation} \par
Let $D=\frac{d}{dx}$, and let $y=x^{p}$. As $ x^{1-\lambda}D=py^{1-\frac{\lambda}{p}}\frac{d}{dy}$, we have
\begin{align}
\big(x^{1-\lambda}D\big)^{n}e^{ax^{p}}&=\bigg(py^{1-\frac{\lambda}{p}}\frac{d}{dy}\bigg)^{n}e^{ay} 
= p^{n}\bigg(y^{1-\frac{\lambda}{p}}\frac{d}{dy}\bigg)^{n}e^{ay} \label{40} \\
&=p^{n}y^{-\frac{n\lambda}{p}}\mathrm{Bel}_{n,\frac{\lambda}{p}}(ay)e^{ay}=p^{n}x^{-n\lambda}\mathrm{Bel}_{n,\frac{\lambda}{p}}(ax^{p})e^{ax^{p}}.\nonumber
\end{align}
Thus, we have 
\begin{equation}
	\Big(x^{1-\lambda}D\Big)^{n}e^{ax^{p}}= p^{n}x^{-n\lambda}\mathrm{Bel}_{n,\frac{\lambda}{p}}(ax^{p})e^{ax^{p}},\quad (n\ge 0). \label{41}
\end{equation}
Therefore, by \eqref{41}, we obtain the following proposition. 
\begin{proposition}
	For $n\ge 0$, we have 
	\begin{displaymath}
		x^{n\lambda}\big(x^{1-\lambda}D\big)^{n}e^{ax^{p}}=p^{n}\mathrm{Bel}_{n,\frac{\lambda}{p}}(ax^{p})e^{ax^{p}},
	\end{displaymath}
	where $D=\frac{d}{dx}$. 
\end{proposition}
From \eqref{8}, we note that 
\begin{align}
&\sum_{k=0}^{n+1}S_{2,\lambda}(n+1,k)(x)_{k}=(x)_{n+1,\lambda}=(x-n\lambda)(x)_{n,\lambda}\label{42}\\
&=(x-n\lambda)\sum_{k=0}^{n}S_{2,\lambda}(n,k)(x)_{k}=\sum_{k=0}^{n}S_{2,\lambda}(n,k)(x-k+k-n\lambda)(x)_{k}\nonumber \\
&=\sum_{k=0}^{n}S_{2,\lambda}(n,k)(x)_{k+1} +\sum_{k=0}^{n}S_{2,\lambda}(n,k)(k-n\lambda)(x)_{k} \nonumber\\
&=\sum_{k=1}^{n+1}S_{2,\lambda}(n,k-1)(x)_{k}+\sum_{k=0}^{n}S_{2,\lambda}(n,k)(k-n\lambda)(x)_{k}\nonumber \\
&=\sum_{k=0}^{n+1}\Big(S_{2,\lambda}(n,k-1)+S_{2,\lambda}(n,k)(k-n\lambda)\Big)(x)_{k}.\nonumber
\end{align}
By \eqref{42}, we get 
\begin{equation}
S_{2,\lambda}(n+1,k)=S_{2,\lambda}(n,k-1)+(k-n\lambda)S_{2,\lambda}(n,k),\label{43}
\end{equation}
where $0 \le k \le n+1$. \par 
\vspace{0.1in}
We prove the next theorem by induction on $n$.
\begin{theorem}
Assume that $f$ is an infinitely differentiable function. Then, for $n\ge 0$, we have 
\begin{displaymath}
\big(x^{1-\lambda}D\big)^{n}f=\sum_{k=0}^{n}S_{2,\lambda}(n,k)x^{k-n \lambda}D^{k}f, 
\end{displaymath}
where $D=\frac{d}{dx}$. 
\end{theorem}
\begin{proof}
The statement is obviously true for $n=0$. Assume that it is true for $n, \,\,(n \ge 0)$.
\begin{align*}
&\big(x^{1-\lambda}D\big)^{n+1}f(x)=\big(x^{1-\lambda}D\big)\sum_{k=0}^{n}S_{2,\lambda}(n,k)x^{k-n \lambda}D^{k}f(x) \\
&=x^{1-\lambda}\sum_{k=0}^{n}S_{2,\lambda}(n,k)\bigg\{(k-n \lambda)x^{k-1-n \lambda}D^{k}f(x)+x^{k-n \lambda}D^{k+1}f(x)\bigg\}\\
&=\sum_{k=0}^{n}S_{2,\lambda}(n,k)\bigg\{(k-n \lambda)x^{k-(n+1) \lambda}D^{k}f(x)+x^{k+1-(n+1) \lambda}D^{k+1}f(x)\bigg\}\\
&=\sum_{k=0}^{n+1}S_{2,\lambda}(n,k)(k-n \lambda)x^{k-(n+1) \lambda}D^{k}f(x)
+\sum_{k=0}^{n+1}S_{2,\lambda}(n,k-1)x^{k-(n+1) \lambda}D^{k}f(x) \\
&=\sum_{k=0}^{n+1}\bigg\{S_{2,\lambda}(n,k)(k-n \lambda)+S_{2,\lambda}(n,k-1)\bigg\}x^{k-(n+1) \lambda}D^{k}f(x)\\
&=\sum_{k=0}^{n+1}S_{2,\lambda}(n+1,k)x^{k-(n+1) \lambda}D^{k}f(x).
\end{align*}
\end{proof}
Let $f(x)=e^{x}$. Then we have 
\begin{align*}
x^{n\lambda}\big(x^{1-\lambda}D\big)^{n}e^{x} &=\bigg(\sum_{k=0}^{n}S_{2,\lambda}(n,k)x^{k}\bigg)e^{x}\\
&=\mathrm{Bel}_{n,\lambda}(x)e^{x}.
\end{align*}
Observe that, for any $\alpha$, we have
\begin{equation}
\Big(x^{1-\lambda}D\Big)^{n}x^{\alpha}=(\alpha)_{n,\lambda}x^{\alpha-n\lambda}. \label{44}
\end{equation} \par
By Leibiniz rule, we get 
\begin{equation}
\Big(x^{1-\lambda}D\Big)^{n}(fg)=\sum_{l=0}^{n}\binom{n}{l}\Big[(x^{1-\lambda}D)^{n-l}f\Big]\Big[(x^{1-\lambda}D)^{l}g\Big].\label{45}	
\end{equation}
From Theorem 3, we note that 
\begin{align}
x^{-(n+m)\lambda}e^{x}\mathrm{Bel}_{n+m,\lambda}(x)&=\Big(x^{1-\lambda}D\Big)^{n+m}e^{x}=\Big(x^{1-\lambda}D\Big)^{n} \Big(x^{1-\lambda}D\Big)^{m}e^{x}\label{46} \\
&= \Big(x^{1-\lambda}D\Big)^{n}\Big(x^{-m\lambda}\mathrm{Bel}_{m,\lambda}(x)e^{x}\Big). \nonumber
\end{align}
By \eqref{45} and \eqref{46}, we get 
\begin{align}
&x^{-(n+m)\lambda}e^{x}\mathrm{Bel}_{n+m,\lambda}(x)
=\Big(x^{1-\lambda}D\Big)^{n}\Big(x^{-m\lambda}\mathrm{Bel}_{m,\lambda}e^{x}\Big)\label{47}\\
&=\sum_{k=0}^{n}\binom{n}{k}\Big[(x^{1-\lambda}D)^{n-k}(x^{-m\lambda}\mathrm{Bel}_{m,\lambda}(x))\Big]\Big[(x^{1-\lambda}D)^{k}e^{x}\Big] \nonumber \\
&=\sum_{k=0}^{n}\binom{n}{k}x^{-k\lambda}\mathrm{Bel}_{k,\lambda}(x)e^{x}\Big[(x^{1-\lambda}D)^{n-k}(x^{-m\lambda}\mathrm{Bel}_{m,\lambda}(x))\Big].\nonumber
\end{align}
On the other hand, 
\begin{align}
&\Big(x^{1-\lambda}D\Big)^{n-k}\Big(x^{-m\lambda}\mathrm{Bel}_{m,\lambda}(x)\Big)=\sum_{j=0}^{m}S_{2,\lambda}(m,j)\Big[(x^{1-\lambda}D)^{n-k}x^{j-m\lambda}\Big]\label{48}	\\
&=\sum_{j=0}^{m}S_{2,\lambda}(m,j)(j-m\lambda)_{n-k,\lambda}x^{j-m\lambda-(n-k)\lambda}=\sum_{j=0}^{m}S_{2,\lambda}(m,j)(j-m\lambda)_{n-k,\lambda}x^{j-(m+n)\lambda+k\lambda} \nonumber\\
&=\sum_{j=0}^{m}S_{2,\lambda}(m,j)\frac{(j)_{m+n-k,\lambda}}{(j)_{m,\lambda}}x^{j-(m+n)\lambda+k\lambda}. \nonumber
\end{align}
By \eqref{47} and \eqref{48}, we get 
\begin{align}
&x^{-(n+m)\lambda}e^{x}\mathrm{Bel}_{n+m,\lambda}(x)\label{49} \\
&=\sum_{k=0}^{n}\binom{n}{k}x^{-k\lambda}\mathrm{Bel}_{k,\lambda}(x)e^{x}\sum_{j=0}^{m}S_{2,\lambda}(m,j)\frac{(j)_{m+n-k,\lambda}}{(j)_{m,\lambda}}x^{j-(m+n)\lambda+k\lambda} \nonumber \\
&=x^{-(m+n)\lambda}e^{x}\sum_{k=0}^{n}\sum_{j=0}^{m}\binom{n}{k}S_{2,\lambda}(m,j)\mathrm{Bel}_{k,\lambda}(x) \frac{(j)_{m+n-k,\lambda}}{(j)_{m,\lambda}}x^{j}\nonumber.
\end{align}
Therefore, by comparing the coefficients on both sides of \eqref{49}, we obtain the following theorem. 
\begin{theorem}
For $m,n\ge 0$, we have 
\begin{equation}
\mathrm{Bel}_{n+m,\lambda}(x)= \sum_{k=0}^{n}\sum_{j=0}^{m}\binom{n}{k}S_{2,\lambda}(m,j)\mathrm{Bel}_{k,\lambda}(x) \frac{(j)_{m+n-k,\lambda}}{(j)_{m,\lambda}}x^{j}.\label{50}
\end{equation}
\end{theorem}
\noindent Taking $x=1$ in \eqref{50}, we have 
\begin{equation}
\mathrm{Bel}_{n+m,\lambda}=\sum_{k=0}^{n}\sum_{j=0}^{m}\binom{n}{k}S_{2,j}(m,j)\mathrm{Bel}_{k,\lambda}\frac{(j)_{n+m-k,\lambda}}{(j)_{m,\lambda}}. \label{51}
\end{equation} \par
From \eqref{14} and \eqref{44}, we note that 
\begin{align}
\big(x^{1-\lambda}D\big)^{n}\mathrm{Bel}_{m,\lambda}(x)&=\Big(x^{1-\lambda}D\Big)^{n}\sum_{k=0}^{m}S_{2,\lambda}(m,k)x^{k} \label{52} \\
&=\sum_{k=0}^{m}S_{2,\lambda}(m,k)(k)_{n,\lambda}x^{k-n\lambda}. \nonumber	
\end{align}
On the other hand, by Leibniz rule \eqref{45} and Theorem 3, we get 
\begin{align}
&\Big(x^{1-\lambda}D\Big)^{n}\mathrm{Bel}_{m,\lambda}(x)=	\Big(x^{1-\lambda}D\Big)^{n}\Big[(e^{-x}x^{m\lambda})(\mathrm{Bel}_{m,\lambda}(x)e^{x}x^{-m\lambda})\Big]\label{53} \\
&=\sum_{k=0}^{n}\binom{n}{k}\Big[(x^{1-\lambda}D)^{n-k}(x^{m\lambda}e^{-x})\Big]\Big[(x^{1-\lambda}D)^{k}\Big(\mathrm{Bel}_{m,\lambda}(x)e^{x}x^{-m\lambda}\Big)\Big]\nonumber \\
&=\sum_{k=0}^{n}\binom{n}{k}\Big[(x^{1-\lambda}D)^{n-k}(x^{m\lambda}e^{-x})\Big]\Big[(x^{1-\lambda}D)^{m+k}e^{x}\Big] \nonumber \\
&=\sum_{k=0}^{n}\binom{n}{k}\Big[(x^{1-\lambda}D)^{n-k}(x^{m\lambda}e^{-x})\Big]x^{-(m+k)\lambda}\mathrm{Bel}_{m+k,\lambda}(x)e^{x}.\nonumber
\end{align}
By \eqref{44}, \eqref{45} and Theorem 3, we easily get 
\begin{align}
(x^{1-\lambda}D)^{n-k}(x^{m\lambda}e^{-x})&=\sum_{j=0}^{n-k}\binom{n-k}{j}\Big[(x^{1-\lambda}D)^{j}x^{m\lambda}\Big]\Big[(x^{1-\lambda}D)^{n-k-j}e^{-x}\Big] \label{54} \\
&=\sum_{j=0}^{n-k}\binom{n-k}{j}(m\lambda)_{j,\lambda}x^{m\lambda-j\lambda}e^{-x}\mathrm{Bel}_{n+k-j,\lambda}(-x)x^{-(n-k-j)\lambda}\nonumber \\
&=\sum_{j=0}^{n-k}\binom{n-k}{j}(m\lambda)_{j,\lambda}x^{m\lambda-n\lambda+k\lambda}\mathrm{Bel}_{n-k-j,\lambda}(-x)e^{-x}.\nonumber	
\end{align}
From \eqref{53} and \eqref{54}, we have 
\begin{align}
&\Big(x^{1-\lambda}D\Big)^{n}\mathrm{Bel}_{m,\lambda}(x)\label{55}\\
&=\sum_{k=0}^{n}\binom{n}{k}x^{-m\lambda-k\lambda}\mathrm{Bel}_{m+k,\lambda}(x)e^{x}\sum_{j=0}^{n-k}\binom{n-k}{j}(m\lambda)_{j,\lambda}x^{m\lambda-n\lambda+k\lambda}\mathrm{Bel}_{n-k-j,\lambda}(-x)e^{-x}\nonumber\\
&=\sum_{k=0}^{n}\sum_{j=0}^{n-k}\binom{n}{k}\binom{n-k}{j}\mathrm{Bel}_{m+k,\lambda}(x)\mathrm{Bel}_{n-k-j,\lambda}(-x)(m\lambda)_{j,\lambda}x^{-n\lambda}.\nonumber
\end{align}
Therefore, by \eqref{52} and \eqref{55}, we obtain the following theorem. 
\begin{theorem}
For $m,n\ge 0$, we have 
\begin{displaymath}
\sum_{k=0}^{m}S_{2,\lambda}(m,k)(k)_{n,\lambda}x^{k} =\sum_{k=0}^{n}\sum_{j=0}^{n-k}\binom{n}{k}\binom{n-k}{j}\mathrm{Bel}_{m+k,\lambda}(x)\mathrm{Bel}_{n-k-j,\lambda}(-x)(m\lambda)_{j,\lambda}.
\end{displaymath}	
\end{theorem}

By \eqref{7} and \eqref{9}, we easily get 
\begin{equation}
(-1)^{n-k}S_{1,\lambda}(n,k)={n \brack k}_{\lambda},\quad(0 \le k \le n).\label{56}	
\end{equation}
Indeed, 
\begin{align}
\sum_{n=0}^{\infty}\langle x\rangle_{n}\frac{t^{n}}{n!}&=\bigg(\frac{1}{1-t}\bigg)^{x}=e_{\lambda}^{-x}\big(\log_{\lambda}(1-t)\big) \nonumber \\
&=\sum_{k=0}^{\infty}(-x)_{k,\lambda}\frac{1}{k!}\big(\log_{\lambda}(1-t)\big)^{k} \label{57}\\
&=\sum_{k=0}^{\infty}(-1)^{k}\langle x\rangle_{k,\lambda}\sum_{n=k}^{\infty}S_{1,\lambda}(n,k)\frac{(-t)^{n}}{n!} \nonumber  \\
&=\sum_{n=0}^{\infty}\bigg(\sum_{k=0}^{n}(-1)^{n-k}S_{1,\lambda}(n,k)\langle x\rangle_{k,\lambda}\bigg)\frac{t^{n}}{n!}.\nonumber
\end{align}
Therefore, by \eqref{57}, we get 
\begin{equation}
\langle x\rangle_{n}=\sum_{k=0}^{n}(-1)^{n-k}S_{1,\lambda}(n,k)\langle x\rangle_{k,\lambda}.\label{58}
\end{equation} \par
Replacing $t$ by $\log_{\lambda}(1+t)$ in \eqref{12}, we get 
\begin{align}
e^{xt}&=\sum_{k=0}^{\infty}\mathrm{Bel}_{k,\lambda}(x)\frac{1}{k!}\big(\log_{\lambda}(1+t)\big)^{k} \label{59}\\
&=\sum_{k=0}^{\infty}\mathrm{Bel}_{k,\lambda}(x)\sum_{n=k}^{\infty}S_{1,\lambda}(n,k)\frac{t^{n}}{n!}\nonumber \\
&=\sum_{n=0}^{\infty}\bigg(\sum_{k=0}^{n}\mathrm{Bel}_{k,\lambda}(x)S_{1,\lambda}(n,k)\bigg)\frac{t^{n}}{n!}.\nonumber
\end{align}
Thus, from \eqref{56} and \eqref{59}, we get 
 \begin{equation}
 x^{n}=\sum_{k=0}^{n}\mathrm{Bel}_{k,\lambda}(x)(-1)^{n-k}{n \brack k}_{\lambda}. \label{60}
  \end{equation} \par
From \eqref{9}, we note that 
\begin{align}
\sum_{k=0}^{n+1}{n+1 \brack k}_{\lambda}\langle x\rangle_{k,\lambda}&=\langle x\rangle_{n+1}=(x+n)\langle x\rangle_{n}=(x+n)\sum_{k=0}^{n}{n \brack k}_{\lambda}\langle x\rangle_{k,\lambda}\label{61} \\
&=\sum_{k=0}^{n}{n \brack k}_{\lambda}(x+k\lambda+n-k\lambda)\langle x\rangle_{k,\lambda}\nonumber\\
&=\sum_{k=0}^{n}{n \brack k}_{\lambda}\langle x\rangle_{k+1,\lambda}+\sum_{k=0}^{n}(n-k\lambda){n \brack k}_{\lambda}\langle x\rangle_{k,\lambda}\nonumber\\
&=\sum_{k=0}^{n+1}{n \brack k-1}_{\lambda}\langle x\rangle_{k,\lambda}+\sum_{k=0}^{n+1}(n-k\lambda){n \brack k}_{\lambda}\langle x\rangle_{k,\lambda}\nonumber \\
&=\sum_{k=0}^{n+1}\bigg({n \brack k-1}_{\lambda}+(n-k\lambda){n \brack k}_{\lambda}\bigg)\langle x\rangle_{k,\lambda}. \nonumber 
\end{align}
Thus, by comparing the coefficients on both sides of \eqref{61}, we get 
\begin{displaymath}
{n+1 \brack k}_{\lambda}= {n \brack k-1}_{\lambda}+(n-k\lambda){n \brack k}_{\lambda},\quad (0 \le k \le n+1). 	
\end{displaymath}

\section{Conclusion}

Here we studied by means of generating functions the degenerate Bell polynomials which are degenerate versions of the Bell polynomials. In more detail, we derived recurrence relations for degenerate Bell polynomials (see Theorems 2, 4, 5, 12), expressions for them that can be derived from repeated applications of certain operators to the exponential functions (see Theorem 3, Proposition 10), the derivatives of them (Corollary 7), the antiderivatives of them (see Theorem 9), and some identities involving them (see Theorems 8. 13). \par
As one of our future projects, we would like to continue to study degenerate versions of certain special polynomials and numbers and their applications to physics, science and engineering as well as to mathematics.

\end{document}